\documentclass[reqno,twoside,final]{amsart}
\usepackage{amsmath,amstext,amsthm,amsfonts,amssymb,amscd}
\newcommand{\etype}[1]{\renewcommand{\labelenumi}{(#1{enumi})}}
\def\eroman{\etype{\roman}}
\newcommand{\C}{\mathbb{C}}

\def\a{\alpha}
\def\sl{\operatorname{sl}}
\def\gl{\operatorname{gl}}
\def\bfi{\bold{i}}
\def\bfj{\bold{j}}

\newcommand{\Image}{{\operatorname{Im}\,}}
\newcommand{\Char}{{\operatorname{Char}\,}}

\newcommand{\tr}{{\operatorname{tr}\,}}

\newcommand{\ssl}{{\operatorname{sl}}}

\newcommand{\GL}{{\operatorname{GL}}}
\newcommand{\ad}{{\operatorname{ad}}}
\newcommand{\PSL}{{\operatorname{PSL}}}

\newcommand{\SL}{{\operatorname{SL}}}

 \DeclareMathOperator{\End}{End}

\theoremstyle{definition}

\newtheorem{cor}{Corollary}
\newtheorem{rem}{Remark}
\newtheorem*{rem*}{Remark}
\newtheorem*{acknow*}{Acknowledgements}
\newtheorem*{examples*}{Examples}
\newtheorem{example}{Example}
\newtheorem{examples}{Example}
\theoremstyle{plain}

\newtheorem{lemma}{Lemma}
\newtheorem{prop}{Proposition}

\newtheorem{thm}{Theorem}
\newtheorem{conjecture}{Conjecture}
\newtheorem{question}{Question}
\newenvironment{proof-sketch}{\noindent{\bf Sketch of Proof}\hspace*{1em}}{\qed\bigskip}
\newenvironment{proof-idea}{\noindent{\bf Proof Idea}\hspace*{1em}}{\qed\bigskip}
\newenvironment{proof-of-lemma}[1]{\noindent{\bf Proof of Lemma #1}\hspace*{1em}}{\qed\bigskip}
\newenvironment{proof-of-prop}[1]{\noindent{\bf Proof of Proposition #1}\hspace*{1em}}{\qed\bigskip}
\newenvironment{proof-of-thm}[1]{\noindent{\bf Proof of Theorem #1}\hspace*{1em}}{\qed\bigskip}
\newenvironment{proof-attempt}{\noindent{\bf Proof Attempt}\hspace*{1em}}{\qed\bigskip}

\begin{document}

\title[Images of Lie polynomials]{The images of Lie
polynomials evaluated on matrices.}
\author{Alexei Kanel-Belov, Sergey Malev, Louis Rowen}

\address{Department of Mathematics, Bar Ilan University,
Ramat Gan, Israel} \email {beloval@math.biu.ac.il}
\address{Einstein Institute of Mathematics, Hebrew University of Jerusalem,
Jerusalem, Israel} \email {malevs@math.huji.ac.il}
\address{Department of Mathematics, Bar Ilan University,
Ramat Gan, Israel}
\email {rowen@math.biu.ac.il}

\thanks{We would like to thank Aner Shalev for interesting and fruitful discussions regarding this
paper}
\thanks{This research was partially supported by the Israel Science Foundation (grant no. 1207/12).}
\thanks{The second named author was partially supported by an ERC grant}
\thanks{The third author would like to thank the University of
Virginia for its support during the preparation of this revision}
\maketitle

\begin{abstract}
 Kaplansky asked about the possible images of a polynomial $f$ in several noncommuting
variables. In this paper we consider the case of $f$   a Lie
polynomial. We describe all the possible images of $f$ in $M_2(K)$
and provide an example of $f$ whose image is the set of
non-nilpotent trace zero matrices, together with 0.  We provide an
arithmetic criterion for this case. We also show that the standard
polynomial $s_k$ is not a Lie polynomial, for $k>2.$
\end{abstract}


\maketitle

\section{Introduction}

A \textbf{Lie polynomial} is an element of the free Lie algebra in
the alphabet $\{ x_i : i \in~I\},$ cf.~\cite[p.~8]{Ra}. In other
words, a Lie polynomial is a sum of Lie monomials $\alpha_j h_j,$
where $h_j$ is a Lie word, built inductively: each letter $x_i$ is a
Lie word of degree~1, and if $h_j, h_k$ are distinct Lie words of
degree $d_j$ and
 $d_k$,
then $[h_j, h_k]$ is a Lie word  of degree $d_j+d_k$.  As usual, a
Lie polynomial is \textbf{multilinear} if each letter appearing in
$f$ appears exactly once in each of its Lie monomials.

This note, which  consists of two parts,  is the continuation of
\cite{BMR1}, in which we considered the question, reputedly raised
by Kaplansky, of the possible image set $\Image f$ of a polynomial
$f$ on   the algebra $M_n(K)$ of $n\times n$ matrices over an
infinite field $K$ of characteristic 0 or prime to $n$. See
\cite{BMR1} for the historical background. Even the case of Lie
identities has room for further investigation. In the first part we
are interested in images of Lie polynomials on  $M_n(K)$, viewed as
a Lie algebra, and thus denoted as $\gl_n(K)$ (or just $\gl_n$ if
$K$ is understood).
 Since $[f,g]$ can be interpreted as $fg-gf$ in the free associative algebra,
  we identify any Lie polynomial with an associative polynomial;
hence. In this way, any set that can arise as the image of a Lie
polynomial on the Lie algebra $\gl_n$ also fits into the framework
of the associative theory of $M_n(K)$, and our challenge here is to
find examples of Lie polynomials that achieve the sets described in
\cite{BMR1,BMR2,BMR3}.

As we shall see, this task is not so easy as it may seem at first
glance. We first consider Lie identities, proving that the
\textbf{standard polynomial} $$s_k: = \sum _{\pi \in S_k}
\operatorname{sgn}(\pi) x_{\pi(1)}\cdots x_{\pi(k)}$$ is not a Lie
polynomial for $k>2$. Then we classify the possible images of Lie
polynomials evaluated on $2\times 2$ matrices and consider the
$3\times 3$ case, based on \cite{BMR1}, where the field $K$ was
required to be quadratically closed, and \cite{M}, where results
were provided over real closed fields, some of them holding more
generally over arbitrary fields.   A key role is played by $\sl_n$,
the Lie algebra of $n\times n$ matrices over $K$ having trace~0.

 In the second part we tie Lie polynomials to the word case.

\section{The images of homogeneous Lie polynomials on $\gl_n$  and $\sl_n$}

We refine Kaplansky's question to the Lie case, and ask:

\begin{question} What is   the possible image set $\Image f$ of a Lie polynomial $f$ on
$\gl_n$ and~$\sl_n$? \end{question}

\begin{question} For which Lie polynomials $f$ of minimal degree do we achieve this image set?
For example, what are the  Lie identities of smallest degree on
$\gl_n$ and~$\sl_n$?
\end{question}


  Even the case of Lie identities is nontrivial, although it has already been
 studied in two important books \cite{Bak,Ra}.
At the outset, the situation for Lie polynomials is subtler than for
regular polynomials, for the simple reason that the most prominent
polynomials in the theory, the standard polynomial $s_n$ and the
Capelli polynomial~$c_n$, turn out not to be
 Lie polynomials.

In order to pass to the associative theory, we make use of the
\textbf{adjoint algebra} $\ad L = \{ \ad_a : L \to L : a \in L\}$
given by $\ad_a (b) = [a,b].$ Note that $$\dim _K (\ad L) < \dim _K
\End _K (L) = (\dim _K L)^2.$$ Also, it is well-known that the map
$a \mapsto \ad_a$ defines a Lie  homomorphism $L \to \ad L$.

We write $ [a_1, \dots, a_t ]$ for  $[a_1, \dots,[a_{t-1},  a_t]],$
and $ [a^{(k)},  a_t  ]$ for $ [a, \dots, a, a_t ]$ where $a$ occurs
$k$ times. By \textbf{ad-monomial} we mean a term  $\alpha \, \ad
x_{i_1} \cdots \ad x_{i_t}$ for some $\alpha\in K$. By
\textbf{ad-polynomial} we mean a sum of ad-monomials.

\begin{rem}\label{corres}
$$\ad _{a_1}\cdots \ad_{a_t}(a) = [a_1, \dots, a_t ,a ].$$ In this way, any ad-monomial corresponds to a Lie monomial,
and thus any ad-polynomial $f(\ad _{x_1}, \dots, \ad _{x_t})$ gives
rise to a Lie polynomial $f( x_1 , \dots,  x_t, y)$ taking on the
same values, and in which $y$ appears of degree~1 in each Lie
monomial in the innermost set of Lie brackets.\end{rem}

Recall that an associative polynomial $f (x_1, \dots, x_k)$ is
\textbf{alternating in the last $m+1$ variables} if  $f$ becomes 0
whenever two of the last $m+1$ variables are specialized to the same
quantity. This yields:

\begin{prop}\label{spec0} Suppose $L$ is a Lie algebra of dimension
$m$, and $f (x_1, \dots, x_k)$ is a multilinear polynomial
alternating in the last $m+1$ variables. Then
$$f(\ad _{x_1}, \dots , \ad _{x_k})(y)$$ corresponds to a Lie
identity of $L$ of degree $\deg f +1$.
\end{prop}
\begin{proof} The alternating property implies
$f (x_1, \dots, x_k)$ vanishes on $\ad L$,
 cf.~\cite[Proposition~1.2.24]{Row},
 so every substitution of $f(\ad _{x_1}, \dots , \ad _{x_k})(y)$
vanishes.
\end{proof}

Since the alternating polynomial of smallest degree is the standard
polynomial~$s_{m+1}$, we have a Lie identity of degree $m+2$ for any
Lie algebra of dimension~$m$. In particular, $\dim (\sl_n) = n^2-1$,
yielding:

\begin{cor}\label{min1} $\sl_n$ satisfies a Lie identity of degree
$n^2+1.$
\end{cor}

 Conversely, we have:

\begin{prop}\label{spec} Suppose $f(x_1, \dots, x_t, y)$ is a Lie
polynomial in which $y$ appears in degree 1 in each of its Lie
monomials. Then $f$ corresponds to an ad-polynomial taking on the
same values on $L$ as $f$.\end{prop}
\begin{proof} In view of Remark~\ref{corres}, it suffices to show
that any Lie monomial $h$ can be rewritten in the free Lie algebra
as a sum of Lie monomials in which $y$ appears (in degree 1) in the
innermost set of Lie brackets. This could be done directly by means
of the Jacobi identity, but here is a slicker argument.

Write $h = [h_1,h_2],$ and we appeal to induction on the degree of
$h$.   $y$ appears say in $h_2$. If $h_1 = y$ then we are done since
$h = -[h_2,y]$ corresponds to $-\ad _{h_2}.$  Likewise if $h_2 = y$.
In general, by induction, $h_1$ corresponds to some ad-monomial $\ad
_{x_{i_1}}\cdots \ad _{x_{i_k}}(y)$  and  $h_2$ corresponds to  some
ad-monomial $\ad _{x_{i_{k+1}}}\cdots \ad _{x_{i_\ell}}(y)$,
 so $[h_1,h_2]$ corresponds
to $\ad _{x_{i_1}}\cdots \ad _{x_{i_k}}((\ad _{x_{i_{k+1}}}\cdots
\ad _{x_{i_\ell}})(y)) =\ad _{x_{i_{1}}}\cdots \ad _{x_{i_\ell}}(y)$
as desired.
\end{proof}

\begin{cor}\label{mustbe} Any homogeneous Lie polynomial of degree $\ge 3$ must be
an identity (viewing the Lie commutator $[a,b]$ as $ab-ba$) of the
Grassmann algebra $G$.
\end{cor}
\begin{proof} Each term includes $[x_i,x_j,x_k],$ which is well
known  to be an identity of~$G$.
 \end{proof}

%
%

Corollary~\ref{min1} gives rise to the following special case of
Question 2:

\begin{question} What is the minimal degree $m_n$ of a Lie identity of
$\sl_n$? \end{question}

By Corollary~\ref{min1}, $m_n \le n^2+1$, and in particular  $m_2
\le 5$. Even the answer $m_2 = 5$  given in \cite[Theorem~36.1]{Ra},
is not easy, although a reasonably fast combinatoric approach is
given in \cite[p.~165]{Bak}, where it is observed that any since any
Lie algebra $L$ satisfying a  Lie identity of degree $<5$ is
solvable, one must have $m_2 \ge 5$, yielding $m_2 = 5$.
\v{S}penko~\cite[Proposition 7.5]{S} looked at this from the other
direction and showed that if $p$ is a Lie polynomial of degree
 $\le 4$ then $\Image p=\ssl_2$.

\begin{examples}\label{usefulex}  $ $

\begin{enumerate}
\eroman \item The standard polynomial $s_2$ itself is a Lie
polynomial.

\item $s_4$ vanishes on $\sl_2$ (viewed inside the associative algebra $M_2(K)$),
since $\sl_2$ has dimension 3. But surprisingly, this is not the
polynomial of lowest degree vanishing on $\sl_2$, as we see next.

\item Bakhturin~\cite[Theorem~5.14]{Bak} points out that  $f = [(x_1x_2 +
x_2 x_1),x_3]$ vanishes on $\sl_2$. In other words, $a_1a_2 + a_2
a_1$ is scalar for any $2 \times 2$ matrices $a_1,a_2$ of trace 0.
Indeed, $a_i^2 $ is scalar for $i=1,2$, implying $a_1a_2 + a_2 a_1$
is scalar  unless $a_1, a_2$ are linearly independent, in which case
$a_1a_2 + a_2 a_1$ commutes with both $a_1 $ and $ a_2$, and thus
again is scalar. But $f$ is not a Lie polynomial, as seen via the
next lemma.
\end{enumerate}
\end{examples}

This discussion motivates us to ask when a polynomial is a Lie
polynomial. Here is a very easy criterion which is of some use.

\begin{lemma}\label{Liep1} Any Lie polynomial which vanishes on $\sl_n$ is an
identity of $\gl_n$.
\end{lemma}
\begin{proof} Immediate, since $\gl_n ' = \sl_n'.$
\end{proof}

The standard polynomial $s_4$   is not a Lie polynomial. Here are
three ways of seeing this basic fact.

  \begin{enumerate} \eroman \item Confront Example \ref{usefulex}
  with the fact that $m_2 = 5$, whereas $\deg s_2 = 4$.

\item A computational approach. We have 15 multilinear Lie
monomials of degree~4, namely $\frac{1}{2}\binom 42 = 3$ of the form
\begin{equation}\label{firsttype} [[x_{i_1}, x_{i_2}],[x_{i_3}, x_{i_4}]]\end{equation}
and $2 \binom 42 = 12$ of the form
\begin{equation}\label{secondtype}[[[x_{i_1}, x_{i_2}],x_{i_3}],
x_{i_4}].\end{equation} But
\begin{equation} [[x_{i_1}, x_{i_2}],[x_{i_3},
x_{i_4}]]  = \ad_{[x_{i_3}, x_{i_4}]} \ad _{x_{i_2}}
(x_{i_1})\end{equation}  \begin{equation*} \qquad \qquad  =
\ad_{x_{i_3}} \ad_{ x_{i_4}}\ad _{x_{i_2}}(x_{i_1})-  \ad_{ x_{i_4}}
\ad_{ x_{i_3}}\ad _{x_{i_2}}(x_{i_1}), \end{equation*} so we can
rewrite the equations \eqref{firsttype} in terms of
\eqref{secondtype}. Furthermore, with the help of the Jacobi
identity, \eqref{secondtype} can be reduced to seven independent Lie
monomials, and one can show that these do not span $s_4$. Even
though this might seem unduly complicated, it provides a general
program to verify that a given polynomial is not Lie.

\item The third approach is simpler  and works for $s_k$, for any $k>2$.

P.M.~Cohn was the first to tie the standard polynomial to the
infinite dimensional Grassmann algebra $G$ with base $e_1, e_2,
\dots$, by noting that $s_k(e_1, \dots, e_k) = k!e_1 \cdots e_k \ne
0$ when $k! \ne 0$.  Rosset~\cite{Ros} used $G$ to reprove the
Amitsur-Levitzki Theorem, and recent  interest has resurged in
studying standard identities via $G$, cf.~\cite{BrPS,P2,I} and also
in the context of Lie algebras~\cite{DPP}.

\end{enumerate}

\begin{thm} The standard polynomial $s_k$ is not a Lie polynomial,
for any $k>2$.
\end{thm}
\begin{proof} Otherwise, by Corollary~\ref{mustbe} it would be an
identity of $G$, contradicting Cohn's observation (taking $\Char K =
0$).
\end{proof}

\subsubsection {A strategy for computing Lie identities}
Lemma~\ref{Liep1} also can be used to determine when a Lie
polynomial $f$ is an identity of $\sl_n$. Indeed, this holds iff it
is a~PI of~$M_n(K)$, and thus of any central simple $K$-algebra, in
particular the symbol algebra $(\a, \beta)$, given by
\begin{equation}\label{symbol1}a^n = \a,\qquad b^n = \beta,\end{equation}
\begin{equation}\label{symbol3}ab = \rho ba\end{equation}
 where $\rho$ is a primitive $n$-root of 1.
This algebra is spanned by the base $\{ a^i b^j : 0 \le i, j <n\}$,
and in some ways this is a better test set for a Lie monomial than
the matrix units, because
\begin{equation}\label{symbolcom}
[a^ib^j, a^k b^\ell] = (\rho ^{jk}-\rho ^{i\ell})a^{i+k} b^{j+\ell}.
\end{equation}

Writing $\ad_{(i,j)}$ for $\ad_{a^ib^j}$ we thus have
$\ad_{(i,j)}(a^k b^\ell)= (\rho ^{jk}-\rho ^{i\ell})a^{i+k}
b^{j+\ell}.$

Let us iterate: Given $(i_1, j_1), \dots, (i_m, j_m)$ which we
denote as $(\bfi, \bfj),$ we also notate $\ad_{(\bfi, \bfj)}: =
\ad_{(i_m, j_m)}\cdots \ad_{(i_1, j_1)}$, and $\bar \bfi_m = \sum
_{u = 1}^m i^m,$ and $\bfj_m = \sum _{u = 1}^m j^m$. An easy
induction argument yields
$$\ad_{(i_m, j_m)} (a^k b^\ell) = \rho^{\prod _m (j_m(\bar j_{m-1}+k )- i_m (\bar i_{m-1}+\ell)}
a^{\bar i_{m}+k}b^{\bar j_{m}+\ell}.$$ Thus, writing a Lie
polynomial as
$$f = \sum \a _{(\bfi, \bfj)}\ad  _{(\bfi, \bfj)},$$
we need to solve $$ \sum \rho^{\prod _m (j_m(\bar j_{m-1}+k )- i_m
(\bar i_{m-1}+\ell)} a^{\bar i_{m}+k}b^{\bar j_{m}+\ell} = 0,$$
taken over $m!$ possible rearrangements  $(i_\pi(1), j_\pi(1)),
\dots, (i_\pi(m), j_\pi(m))$ of  $(\bfi, \bfj).$

In the generic case, we can take $\a, \beta$ to be commuting
indeterminates.
 Now writing a Lie polynomial  as $\sum _u c_u h_u$ for $c_u \in K$
and Lie monomials $h_u$, we evaluate all $c_u \ad_{(i_1,j_1)}\cdots
\ad_{(i_m,j_m)}$ on all $a^k b^\ell$, noting that there  are
$(n^2)^{m+1}$ equations to solve, and  we need the minimal $m$ such
that the rank of the coefficient matrix is less than the number of
variables. The nontrivial solutions provide the Lie identities. This
can be done on the computer for any given $n$, although we do not
have a result for general $n$.

\subsection {The case $n = 2$}


Recall from Corollary~\ref{min1} that there is a Lie identity of
degree~5.

\begin{thm}
\label{main} If $f$ is a homogeneous Lie polynomial evaluated on the
matrix ring $M_2(K)$, where $K$ is an algebraically closed field,
then $\Image f$ is either $\{0\}$, or  $K$ (the set of scalar
matrices), or the set of all non-nilpotent matrices having trace
zero, or $\ssl_2(K)$, or $M_2(K)$.
\end{thm}
\begin{rem}
The case of scalar matrices in Theorem \ref{main} is possible only
when $\Char K=2$, and the last case $M_2(K)$ is possible only if
$\deg f=1$.
\end{rem}
\begin{proof-of-thm}{\ref{main}}
According to  \cite[Theorem~1]{BMR1} the image of $f$ must be either
 $\{0\}$, or $K$, or the set of all non-nilpotent matrices having
trace zero, or $\ssl_2(K)$,  or  a dense subset of $M_2(K)$ (with
respect to Zariski topology). Note that if at least one matrix
having nonzero trace   belongs to the image of $f$ then $\deg f=1$
and thus $\Image f=M_2(K)$.
\end{proof-of-thm}
\begin{thm} For  any algebraically closed field $K$ of characteristic  $\ne 2,$
 the image of  any   Lie polynomial $f$ (not necessarily homogeneous) evaluated on $\ssl_2(K)$
  is either $\ssl_2(K)$, or $\{0\}$, or the set of trace zero non-nilpotent matrices.
\end{thm}
\begin{proof} For $f$ not a PI,
 we can write $f=f_j+f_{j+1}+\dots+f_d$, where each $f_i$
  is a homogeneous Lie polynomial of degree $i$, and $f_d$ is not PI.
 Therefore for any $c\in K$ we have
 $$f(cx_1,cx_2,\dots,cx_m)=c^jf_j+\dots+c^df_d.$$
 Since $f_d$ is not $PI$, we can take specializations of $x_1,\dots, x_m$ for which $\det(f_d)\neq 0$.
Fixing these specializations, we consider
$\det(c^jf_j+\dots+c^df_d)$
 as a polynomial in $c$  of degree $j + \cdots + d$.
Since the leading coefficient is not zero and $K$ is algebraically
closed, its image is $K$.
 Thus for any $k\in K$ there exist $x_1,\dots,x_m$ for which $\det(f)=k$.
 Hence (for $\Char K\neq 2$) any matrix with  nonzero eigenvalues $\lambda$ and~$-\lambda$ belongs to $\Image f$.
 Therefore $\Image f$ is either $\ssl_2$ or the set of trace zero non-nilpotent matrices.

\end{proof}

Let us give examples of Lie polynomials having such images:
\begin{example}\label{badex}
\
If $\Char K=2$, then   $\Image f = K$ also is possible:
We take
$$f(x,y,z,t)=[[x,y],[z,t]].$$
Any value of $f$ is the Lie product of two trace zero matrices
$s_1=[x,y]$ and $s_2=[z,t]$. Both can be written as
$s_i=h_i+u_i+v_i$, where the $h_i$ are diagonal trace zero
matrices (which are scalar since $\Char K=2$), the $u_i$ are
proportional to $e_{12}$, and the $v_i$ are proportional to
$e_{21}$. Thus $[s_1,s_2]=[u_1,v_2]+[u_2,v_1]$ is scalar.

 Over an arbitrary field,  $\Image f$
can indeed be equal to $\{0\}$, or $K$, or the set of all
non-nilpotent matrices having trace zero, or $\ssl_2(K)$, or
$M_2(K)$.

\begin{enumerate} \eroman \item $\Image x=M_2(K)$.

\item  $\Image [x,y]=
\ssl_2$.

\item Next, we construct  a Lie polynomial whose image evaluated on
$\ssl_2(K)$ is the set of all non-nilpotent matrices having trace
zero. We take the multilinear polynomial $h(u_1,\dots,u_8)$
constructed in \cite{DK2} by Drensky and Kasparian which is central
on $3\times 3$ matrices. Given    $2\times 2$ matrices $x_1,\dots,
x_9$ we consider the homogeneous Lie polynomial
$$f(x_1,\dots,x_9)=h(\ad_{ [x_9 ,x_9 ,\dots, x_9,x_1]},\ad_{x_2},\ad_{x_3},\dots,\ad_{x_8})(x_9).$$
For any $2\times 2$ matrix $x$, $\ad_x$ is a $3\times 3$ matrix
since $\ssl_2$ is $3$-dimensional; hence, for any values of $x_i$,
the value of $f$ has to be proportional to $x_9$. However for $x_9$
nilpotent, this must  be zero, since $[x^{(3)},y]=0$ for any
$y\in\ssl_2(K)$ if $x$ is nilpotent. (When we open the brackets we
have the sum of $8$ terms and each term equals $x^kyx^{3-k}$. But
for any integer $k$, either $k\geq 2$ or $3-k\geq 2$.) Thus the
image of $f$ is exactly the set of non-nilpotent trace zero
matrices.

Another example of a homogeneous Lie polynomial with no nilpotent
values is $f(x,y)=[[[x,y],x],[[x,y],y]]$. (See
\cite[Example~4.9]{BGKP} for details.)

\end{enumerate}

\end{example}

\subsection {The case $n = 3$}

New questions arise concerning the possible evaluation of Lie
polynomials on $M_n(K)$.

According to \cite[Theorem 3]{BMR2}, if $p$ is a homogeneous polynomial with trace vanishing image, then
$\Image p$ is one of
the following:

\begin{itemize}
\item \{0\}, \item  the set of scalar matrices (which can occur
only if $\Char K=3$),
 \item a dense subset of  $\sl_3(K)$, or \item
the set of $3$-scalar matrices, i.e., with
eigenvalues $(c,c\omega,c\omega^2)$,
where $\omega$ is our cube root of $1$.
\end{itemize}

Drensky~and Rashkova~\cite{DR} have found several  identities of
$\sl_3$ of degree 6, but they cannot be Lie polynomials, since
otherwise they would be identities of $\gl_3$ and thus a multiple of
$s_6,$ which is not a Lie polynomial. Thus, one must go to higher
degree.

In the associative case, the fact that the generic division algebra
has a $3$-central element implies that there is a homogeneous
$3$-central polynomial $f$ for $M_3(K),$ i.e., all of whose values
take on eigenvalues $c, \omega c, c\omega ^2$, where $\omega$ is a
cube root of $1$. But any matrix with these eigenvalues is either
scalar or has trace $0$. This leads us to the basic questions needed
to complete the case $n=3$:

\begin{question} Is there a Lie polynomial $f$ whose values are dense on
$\sl_3 (\mathbb C)$ but do not  take on all  values?
\end{question}

\begin{question} Is there a Lie polynomial $f$ whose values on
$\sl_3$ all take on eigenvalues $c, \omega c, c\omega ^2$, where
$\omega$ is a primitive cube root of $1$?
\end{question}

\subsection{A Group theoretical question and its relation to the Lie theoretical problem}
\

Let $w$ be an element of the free group of $m$ letters
$x_1,x_2,\dots x_{m-1}$ and $x_m$. Given a group $G$, we consider
the map $f_{w, G}:G^m \rightarrow G$ corresponding to the word $w$.
This map is called a {\it word map}, which for convenience we also
notate as $w$ instead of $f_{w, G}$. There is a group conjecture
(see \cite[Question~2]{KBKP} for the more general case):

\begin{conjecture}\label{word-PSL} If the field $K$ is algebraically closed
of characteristic $0$, then the image of any nontrivial group word
$w(x_1,\dots,x_m)$ on the projective linear group $\PSL_2(K)$  is
$\PSL_2(K)$.
\end{conjecture}

\begin{rem} Note that if one takes the group $\SL_2$ instead of $\PSL_2$, Conjecture~\ref{word-PSL}
fails, since the matrix $-I+e_{12}$ does not belong to the image of
the word map $w=x^2.$
\end{rem}

\begin{examples}
When $\Char K=p>0$, the image of the word map $w(x)=x^p$ evaluated
on $\PSL_2(K)$ is not $\PSL_2(K)$. Indeed, otherwise the matrix
$I+e_{12}$ could be written as $x^p$ for $x\in\PSL_2(K)$.
 If the eigenvalues of $x$ are equal, then $x=I+n$ where $n$ is nilpotent. Therefore $x^p=(I+n)^p=I+pn=I.$
 If the eigenvalues of $x$ are not equal, then $x$ is diagonalizable and therefore $x^p$ is also diagonalizable, a contradiction.
\end{examples}

%

\begin{lemma}[Liebeck, Nikolov, Shalev, cf.~also \cite{G} and
\cite{Ban}] $\Image w$ contains all matrices from $\PSL_2(K)$ which
are not unipotent.
\end{lemma}

\begin{proof}
According to \cite{B} the image of the word map $w$ must be Zariski
dense in $\SL_2(K)$. Therefore the image of $\tr w$ must be Zariski
dense in $K$. Note that $\tr w$ is a homogeneous rational function
and $K$ is algebraically closed. Hence, $\Image (\tr w)=K.$ For any
$\lambda\neq\pm 1$ any matrix with eigenvalues $\lambda$ and
$\lambda^{-1}$ belongs to the image of $w$ since there is a matrix
with trace $\lambda+\lambda^{-1}$ in $\Image w$ and any two matrices
from $\SL_2$ with equal trace (except trace $\pm 2$) are similar.
Note that the identity matrix $I$ belongs to the image of any word
map.
\end{proof}

However the question whether one of the matrices $(I+e_{12})$ or
$(-I-e_{12})$ (which are equal in $\PSL_2$) belongs to the image of
$w$ remains  open. We conjecture that $I+e_{12}$ must belong to
$\Image w$. Note that if there exists $i$ such that the degree of
$x_i$ in $w$ is $k\neq 0$ then we can consider all $x_j=I$ for
$j\neq i$ and $x_i=I+e_{12}$. Then the value of $w$ is
$(I+e_{12})^k=I+ke_{12}$ and this is a unipotent matrix since $\Char
K=0$, and thus $\Image w=\PSL_2(K)$. Therefore it is interesting to
consider word maps $w(x_1,\dots,x_m)$ such that the degree of each
$x_i$ is zero.

This is why   Conjecture \ref{word-PSL} can be reformulated as
follows:

\begin{conjecture}\label{word-PGL}
Let $w(x_1,\dots,x_m)$ be a group word whose degree at each $x_i$ is
$0$. Then the image of $w$ on $G=\GL_2(K)/\{\pm 1\}$ must be
$\PSL_2(K)$.
\end{conjecture}

One can consider matrices $z_i=\frac{x_i}{\sqrt{\det x_i}}$ and note
that $w(z_1,\dots,z_m)=w(x_1,\dots,x_m)$.

For Conjecture \ref{word-PGL} we take $y_i=x_i-I$. Then we can open
the brackets in
$$w(1+y_1,1+y_2,\dots,1+y_m)=1+f(y_1,\dots,y_m)+g(y_1,\dots,y_m),$$
where $f$ is a homogeneous Lie polynomial of degree $d$, and $g$ is
the sum of terms of degree greater than $d$. Therefore it is
interesting to investigate the possible images of Lie polynomials,
whether it is possible that the image of $l$ does not contain
nilpotent matrices. Unfortunately we saw such an example
(\ref{badex}, although its degree must be at least 5 by Spenko
\cite[Lemma~7.4]{S}. More general questions about surjectivity of
word maps in groups and polynomials in algebras are considered in
\cite{KBKP}.

\begin{rem}
Our next theorem describes the situation in which the trace
vanishing polynomial does not take on nonzero nilpotent values. It
implies that any nontrivial word map $w$ evaluated on $\PSL_2$ is
not surjective iff its projection to $\ssl_2$  given by $\ssl_2:\
x\mapsto x-\frac{1}{2}\tr x$  is a multiple of any prime divisor of
$\det(\pi(w)).$ This might help in answering Conjecture
\ref{word-PSL}.
\end{rem}
\begin{thm}\label{nonilp}
Let $f(x_1,\dots,x_m)$ be a trace vanishing polynomial, evaluated on
$M_n(K[\xi ]). $ Let $\bar f = f(y_1,  \dots, y_m).$ Then $f$ takes
on no nonzero nilpotent values on any integral domain containing
$K$, iff each prime divisor $d$ of $\det(\bar f)$ also divides each
entry of~$\bar f$.
\end{thm}
\begin{proof} $(\Rightarrow)$ If some prime divisor $d$ of $\det(\bar
f)$ does not divides $\bar f$, then $\bar f$ does not specialize to
$0$ modulo $d$. Therefore we have a nonzero matrix in the image of
$f$ which has determinant zero and also trace zero, and thus is
nilpotent, a contradiction.

 $(\Leftarrow)$  Assume that $f$ takes on a
nonzero nilpotent value  over some extension integral domain of $K$.
Thus $\det \bar f$ goes to 0 under the corresponding specialization
of the $\xi_{i,j}^k$, so some prime divisor $d$ of $\det(\bar f)$
goes to $0$, and  $\bar f$ is not divisible by~$d$.
\end{proof}

\end{document}